\documentclass{daj}

\usepackage{amssymb, amsmath, amsfonts,amsthm,amssymb,amscd, hyperref, verbatim,graphicx,color,multirow,booktabs, caption,tikz,tikz-cd, mathdots,bm,tikz-cd}
\usepackage[shortlabels]{enumitem}
\usepackage[capitalise, noabbrev]{cleveref}

 \newtheorem{theorem}{Theorem}[section]
\newtheorem{lemma}[theorem]{Lemma}

 \newtheorem{corollary}[theorem]{Corollary}

 \newtheorem{theoremX}{Theorem} 
 
 \newtheorem{corollaryX}[theoremX]{Corollary} 
 
      \theoremstyle{definition}
     \newtheorem*{definition}{Definition}
     \newtheorem{example}[theorem]{Example}
     \theoremstyle{remark}
     \newtheorem{remark}[theorem]{Remark}

\newcommand{\Alt}{\mathrm{Alt}}
\newcommand{\Sym}{\mathrm{Sym}}
\newcommand{\Sol}{\mathrm{Sol}}
\newcommand{\Aut}{\mathrm{Aut}}
\newcommand{\Out}{\mathrm{Out}}
\newcommand{\Irr}{\mathrm{Irr}}
\newcommand{\PSL}{\mathrm{PSL}}

\newcommand{\Sp}{\mathrm{Sp}}
\newcommand{\SL}{\mathrm{SL}}
\newcommand{\GL}{\mathrm{GL}}

\newcommand{\C}{\mathbf{C}}
\newcommand{\kk}{\mathbf{k}}
 \newcommand{\rr}{\mathbf{r}}
\newcommand{\D}{\mathbf{D}}
\newcommand{\Z}{\mathbf{Z}}

\newcommand{\Di}{\mathbf{R}}

\renewcommand{\a}{\boldsymbol{\varepsilon}}
\newcommand{\iot}{\boldsymbol{\iota}} 

\dajAUTHORdetails{%
  title = {Quasirandom and Quasisimple Groups}, 
  author = {Marco Barbieri and Luca Sabatini},
  plaintextauthor = {Marco Barbieri, Luca Sabatini},
    runningauthor = {M. Barbieri and L. Sabatini},
  keywords = {Quasirandom groups, quasisimple groups.},
}   

\dajEDITORdetails{%
   year={2025},
   number={21},
   received={2 May 2024},   
   published={2 October 2025},  
   doi={10.19086/da.144893},       
}   

\begin{document}

\begin{frontmatter}[classification=text]

\title{Quasirandom and Quasisimple Groups} 

\author[mb]{Marco Barbieri\thanks{Supported by the Slovenian Research Agency program P1-0222 and grant J1-50001, and member of the Italian GNSAGA INdAM research group.}}
\author[ls]{Luca Sabatini\thanks{Supported by the Royal Society.}}

\begin{abstract}
Fix $\varepsilon > 0$.
We say that a finite group $G$ is $\varepsilon$-quasirandom if every nontrivial irreducible complex representation of $G$
has degree at least $|G|^\varepsilon$.
In this paper, we give a structure theorem for large $\varepsilon$-quasirandom groups,
and we completely classify the $\frac{1}{5}$-quasirandom groups.
\end{abstract}
\end{frontmatter}

\section{Introduction}
 
Informally, a finite group is called {\itshape quasirandom} if its nontrivial complex irreducible representations have ``large'' degree,
where the precise meaning of ``large'' depends on context.
The terminology was introduced explicitly in \cite{Gow08},
where the name arises because dense Cayley graphs in quasirandom groups are
quasirandom graphs in the sense of Chung, Graham, and Wilson \cite{CGW89}.
This property had already been used implicitly to construct expander graphs in \cite{Gam02,SX91}.
After Gowers' groundbreaking paper, quasirandom groups have been at the center of spectacular development.
Bourgain and Gamburd \cite{BG08} used the quasirandomness of $\SL_2(p)$ to show that random Cayley graphs on these groups
are expanders with high probability.
Their work has inspired an impressive series of articles, culminating in \cite{BGGT15}.
In a more subtle way, quasirandomness played a key role in Helfgott's work \cite{Hel08} and generalizations,
most notably \cite{BGT11,EMPS23,PS16}.
Tao dedicated an entire chapter of his book \cite{Tao15} to quasirandom groups.

As said, being quasirandom is a quantitative property.
 For applications to expansion in Cayley graphs,
 one needs that every nontrivial irreducible representation has degree at least $|G|^\varepsilon$,
  for some constant $\varepsilon>0$ independent of the size of $G$.
  This motivates the following definition:
  
  \begin{definition}
Let $\varepsilon>0$.
A finite group $G$ is {\itshape $\varepsilon$-quasirandom} if every nontrivial irreducible complex representation of $G$
has degree at least $|G|^\varepsilon$.
\end{definition} 
  
It is well known that the finite simple groups of Lie type are $\varepsilon$-quasirandom
for some $\varepsilon>0$ which depends only on the rank \cite{LS74},
a fact that is exploited in all the papers we mentioned above.
A more general discussion on quasirandomness is taken by Babai, Nikolov and Pyber \cite{BNP08}, and Nikolov and Pyber \cite{NP11}.
The celebrated {\itshape Gowers' trick} \cite{NP11} says that,
if $G$ is $\varepsilon$-quasirandom and $S \subseteq G$ is a large subset, namely $|S| \geq |G|^{1-\frac{\varepsilon}{3}}$, then $S^3=G$.
This fact, as well as other natural invariants of $G$, can be used to characterize quasirandomness.
Such equivalences look particularly nice when expressed in the exponential regime (see Theorem \ref{thChar} below).

The algebraic structure of quasirandom groups remains unexplored.
 The present article was born from our need to answer the question:
 {\itshape what, concretely, is a quasirandom group?}
In the first main result we describe the large $\varepsilon$-quasirandom groups, for any fixed $\varepsilon >0$.
 They are precisely the groups whose quotient by the solvable radical
  is a relatively large direct product of boundedly many finite simple groups of Lie type of bounded rank.
 
\begin{theoremX} \label{thm:A}
	Fix $\varepsilon>0$.
	If $G$ is an $\varepsilon$-quasirandom group, then $|G/\Sol(G)| \geq |G|^\varepsilon$.
	If $|G|$ is sufficiently large depending on $\varepsilon$,
	then $G/\Sol(G)$ is a direct product of boundedly many finite simple groups of Lie type of bounded rank.
	
	Conversely, suppose that $G$ is perfect, $|G/\Sol(G)| \geq |G|^\varepsilon$, and that $G/\Sol(G)$ is $\varepsilon$-quasirandom.
	If $|G|$ is sufficiently large depending on $\varepsilon$,
	then $G$ is $\frac{\varepsilon^2}{3}$-quasirandom.
\end{theoremX} 

For the second main result, we study the groups that are $\varepsilon$-quasirandom with a relatively large $\varepsilon$.
It is a widespread belief that $(\PSL_2(q))_{q \geq 4}$ is the ``most quasirandom'' family of groups \cite{MO1,MO2},
but not much is known.
Not only is this actually the case, but we prove that the vast majority of highly quasirandom groups are quasisimple,
shedding a good light on the issue.
	
	\begin{theoremX} \label{thm:B} 
		Let $G$ be a $\frac{1}{5}$-quasirandom group. Then one of the following holds:
		\begin{enumerate}[$(a)$]
		\item \label{thm:B:a} $G \cong (\mathbb{F}_3)^3 \rtimes \SL_3(3)$;
		\item \label{thm:B:b} $G$ is a perfect central extension of $(\mathbb{F}_{2^f})^4 \rtimes \Sp_4(2^f)$, $f \geq 6$;
		\item \label{thm:B:c} $G$ is quasisimple.
		\end{enumerate}
	\end{theoremX} 
	
	The low dimensional representations of the quasisimple groups are known \cite{Tie03},
	thus it is an easy (although laborious) task to recover the full list of the $\frac{1}{5}$-quasirandom groups from Theorem \ref{thm:B}.
	An interesting consequence is the following:
	
	\begin{corollaryX} \label{cor:3/14}
		Every $\frac{3}{14}$-quasirandom group is quasisimple.
	\end{corollaryX}
	
	The constant $\frac{3}{14}$ is best possible, precisely because of the family in Theorem \ref{thm:B}\ref{thm:B:b}.
	On the other hand, the value $\frac{1}{5}$ is the result of a compromise,
	as with our machinery and more work one can obtain a deeper classification.
	There are some reasons to prefer $\frac{1}{5}$:
	first, it is small enough to provide Corollary \ref{cor:3/14}.
	Second, it makes the theorem sharp in some sense,
	because there exist infinitely many nonquasisimple groups which are $(\frac{1}{5} -\varepsilon)$-quasirandom,
	for any $\varepsilon>0$ (see Section \ref{sec3}).
	It also makes the proof a little cleaner,
	because every $\frac{1}{5}$-quasirandom group has a unique maximal normal subgroup (Lemma \ref{lemUniqueM}).
	We point out that similar methods are useful to study quasirandomnes at other regimes different from the exponential,
	which could be of interest as well (see also Remark \ref{remAtt}).
	The shallow version of Theorem \ref{thm:B} is the following:
	
	\begin{corollaryX} \label{cor:1/3}
		The only $\frac{1}{3}$-quasirandom groups are the sporadic simple groups $J_1$ and $O'N$.
	\end{corollaryX}

	We will actually prove Corollary \ref{cor:1/3} as a ``baby'' Theorem \ref{thm:B}.
	The philosophy behind the proofs is that quasirandom groups are pretty rigid,
	and lend themselves very little to nontrivial extensions, especially the noncentral ones.
	To make this idea really say something, we use various properties of the finite simple groups and their representations.
	Particularly for Theorem \ref{thm:B}, we need a detailed (i.e. quantitative) analysis of quasirandom degrees,
	modular and projective representations.
	
	The paper is organized as follows.
	In Section \ref{sec2} we introduce the (exponential) quasirandom degree $\a(G)$ of a finite group,
	and give basic results about quasisimple groups, modular and projective representations.
	In Section \ref{sec3}, we provide some highly quasirandom nonquasisimple groups:
	they are of affine type as in Theorem \ref{thm:B} \ref{thm:B:a}-\ref{thm:B:b}.
	We prove the main theorems in Section \ref{sec4}.
	Fundamental to our discussion is the knowledge of the quasirandom degrees of the finite simple groups,
	so we collect these data in the Appendix \ref{appendix}.

    \section{Preliminaries} \label{sec2}

\subsection{Minimal degree of a nontrivial representation} 

In this paper, every group is finite.
For a nontrivial finite group $G$,
we denote by $\D_0(G)$ the minimal degree of a nontrivial irreducible complex representation of $G$, that is,
$$ \D_0(G) \> := \> \min \left\{ \dim(\pi) \mid \pi \in \Irr(G) \setminus \{1\} \right\}  . $$

Observe that the veracity of Corollaries \ref{cor:3/14} and \ref{cor:1/3} implies that $\D_0(G) < |G|^{0.3464}$ for every finite group $G$, and $\D_0(G) < |G|^{0.2143}$ if $G$ is not quasisimple.

We start with the following, which appears as \cite[Ex. 1.3.2]{Tao15}.

\begin{lemma} \label{lemExt1}
	Let $N \lhd G$ be a nontrivial proper normal subgroup.
	Then
	$$ \D_0(G/N) \> \geq \> \D_0(G) . $$
	If the previous inequality is strict, then
	$$ \D_0(G) \> \geq \> \D_0(N) . $$
\end{lemma}
\begin{proof}
Each nontrivial irreducible representation of $G/N$ provides a nontrivial representation of $G$.
The first inequality follows.
For the second one, suppose there exists $\pi \in \Irr(G)$ such that $\dim(\pi) = \D_0(G) < \D_0(G/N)$. 
Observe that the restriction $\pi_{\downarrow_N}$ of $\pi$ to $N$ is nontrivial: otherwise $N \subseteq Ker(\pi)$,
and $\pi$ would induce a nontrivial irreducible representation of $G/N$ whose degree is $\dim(\pi)$,
against the minimality of $\D_0(G/N)$.
Therefore, $\pi_{\downarrow N}$ contains a nontrivial irreducible representation of $N$,
and so $\dim(\pi) = \dim ( \pi_{\downarrow_N} ) \geq \D_0(N)$.
\end{proof} 

\begin{lemma} \label{lemActions}
If $G$ acts nontrivially on a finite set $\Omega$, then $\D_0(G) \leq |\Omega|-1$.
\end{lemma}
\begin{proof}
Recall that $\Sym(n)$ has a natural embedding into $\GL_n$, and that any permutation representation contains the trivial representation,
because the diagonal subspace is fixed by every element of $G$.
\end{proof}

We can obtain some useful bounds by considering the natural actions of $G$.

\begin{corollary} \label{corActions}
The following inequalities hold:
\begin{enumerate}[$(i)$]
	\item if $H < G$ is a proper subgroup, then
	\[ |G:H| \, \geq \, \D_0(G/\mathrm{Core}_G(H)) +1 \,; \]
	\item if $N \lhd G$ is a noncentral normal subgroup, then
	\[ |N| \, \geq \, \D_0(G/\C_G(N)) +2 \,; \]
	\item if $C \subseteq G$ is a nontrivial conjugacy class, then
	\[ |C| \, \geq \, \D_0(G) +1 \,.\]
\end{enumerate}
\end{corollary} 
\begin{proof}
	Use Lemma \ref{lemActions} on:
	the action of $G/\mathrm{Core}_G(H)$ on the right cosets of $H$ by multiplication;
	the action of $G/\C_G(N)$ on $N \setminus \{1\}$ by conjugation;
	the action of $G$ on $C$ by conjugation.
\end{proof} 

We would like to mention some connections between $\D_0(G)$ and other important parameters associated to a finite group:
the number of conjugacy classes $\kk(G)$, the representation growth $\rr_n(G)$,
and the minimal index of a proper subgroup, i.e. $\iot(G) := \min_{H<G} |G:H|$.

\begin{remark}
By standard results in representation theory, we can write
$$ |G| =  \sum_{\pi \in \Irr(G)} \dim(\pi)^2 \geq 1 + (\kk(G)-1) \D_0(G)^2 . $$
Let $G$ be an $\varepsilon$-quasirandom group.
It follows that $\varepsilon < \frac{1}{2}$, and $\kk(G) \leq |G|^{1-2\varepsilon}$.
Moreover, the number $\rr_n(G)$ of the inequivalent irreducible representations of degree $n$ is polynomially bounded,
namely $\rr_n(G) \leq n^{\varepsilon^{-1} -2}$ for all $n$ (see \cite[Lemma 25]{Sab23}).
\end{remark} 

\begin{remark} 
		Using the Classification of the Finite Simple Groups,
		Nikolov and Pyber \cite{NP11} have proved that there is an absolute constant $c$ such that
		$$ \iot(G) \> \leq \> c \> \D_0(G)^2 $$
		for every finite perfect group $G$.
		A converse inequality, $\D_0(G) < \iot(G)$, follows from the previous lemmas.
		Lev \cite{Lev92} proved that $\iot(G) \leq |G|^{0.5}$ for every finite group $G$ whose order is not a prime
		(equality is achieved by groups of order $p^2$).
	\end{remark}

\subsection{Quasirandom degree}

Let $G$ be a nontrivial finite group. We introduce the {\itshape (exponential) quasirandom degree} of $G$ by
$$ \a(G) \> := \> \log_{|G|} \D_0(G) . $$
By definition, $\D_0(G)=|G|^{\a(G)}$,
and $\a(G)$ is the maximum $\varepsilon \in [0,\frac{1}{2})$ such that $G$ is an $\varepsilon$-quasirandom group.
Observe that $\a(G) >0$ if and only if $G$ is perfect.
We now report the theorem of Nikolov and Pyber mentioned in the Introduction:

\begin{theorem}[Theorem 3.1 in \cite{NP11}] \label{thChar}
Let $G$ be a finite perfect group.
The following statements are roughly equivalent,
in the sense that if one statement holds for a constant $\varepsilon_i>0$,
then all others hold with constant $\varepsilon_j=c_{i,j} \varepsilon_i$, for some $c_{i,j}>0$.
\begin{itemize}
\item $G$ is $\varepsilon_1$-quasirandom;
\item every proper subgroup of $G$ has index at least $|G|^{\varepsilon_2}$;
\item for every subset $S$ of size at least $|G|^{1-\varepsilon_3}$ we have $S^3=G$;
\item every product-free subset of $G$ has size at most $|G|^{1-\varepsilon_4}$.
\end{itemize}
\end{theorem}

Let $G$ be a finite group and $N \lhd G$ a proper normal subgroup. We can write
$$ \a(G) \> = \> \a(G/N) \cdot \frac{\log \D_0(G)}{\log \D_0(G/N)} \cdot \frac{\log|G/N|}{\log|G|} . $$
(We remark here that, when we consider a ratio of logarithms, the base can be chosen arbitrarily.)
It is immediate from the equality above and Lemma \ref{lemExt1} that $\a(G/N) \geq \a(G)$, with equality if and only if $N$ is trivial.
We will use this fact throughout the article without mentioning it.

We now give a simple but very useful inequality.
Let $N \lhd G$ be a proper normal subgroup. 
By the definition of quasirandom degree, and Lemma \ref{lemExt1},
$$ |G| \> = \> \D_0(G)^{1/\a(G)} \> \leq \> \D_0(G/N)^{1/\a(G)} . $$
Hence, substituting $|G|=|G/N||N|$, we obtain
\begin{equation}
\label{eqN}
	|N| \> \leq \> \frac{\D_0(G/N)^{1/\a(G)}}{|G/N|} .
\end{equation} 
What makes (\ref{eqN}) so useful is that in many occasions we barely know that $\a(G)$ is large, but we have a good knowledge of $G/N$.
The next lemma provides a gap between $\a(G)$ and $\a(G/N)$ when $N$ is abelian but noncentral,
and it greatly simplifies the proof of Theorem \ref{thm:B}.

\begin{lemma} \label{lemImprove}
	Let $N \lhd G$ be a noncentral abelian normal subgroup.
	Then
	$$ \a(G/N) \> > \> \frac{\a(G)}{1-\a(G)} . $$
\end{lemma}
\begin{proof}
	Since $N$ is abelian, we have $N \subseteq \C_G(N) \neq G$.
	Combining Lemma \ref{lemExt1}, Corollary \ref{corActions}(ii), and (\ref{eqN}), we write
	\[\begin{split}
		\D_0(G/N) 
		&\leq \D_0(G/\C_G(N))
		\\&< |N|
		\\&\leq \frac{\D_0(G/N)^{1/\a(G)}}{|G/N|} . 
	\end{split} \, \] 
	Writing $\D_0(G/N)=|G/N|^{\a(G/N)}$, we obtain
	$$ |G/N|^{\a(G/N)} < |G/N|^{\frac{\a(G/N)}{\a(G)}-1} , $$ 
	and the result is proven by comparing the exponents. 
\end{proof}  

Lemma \ref{lemImprove} may fail if $N$ is central:
for instance, we have
$$ \a(\SL_2(7)) \approx 0.1889, \quad \a(\PSL_2(7)) \approx 0.2144, $$
and 
$$ \a(\PSL_2(7)) \> < \> \frac{\a(\SL_2(7))}{1-\a(\SL_2(7))} . $$
Such a phenomenon suggests that quasisimple groups are strong candidates looking for highly quasirandom groups,
because $\a(G)$ can be very close to $\a(G/\Z(G))$.

We now estimate the quasirandom degree of a direct product.

\begin{lemma} \label{lemDP}
The following holds:
\begin{itemize}
\item[(i)] If $\a(A) \leq \a(B)$, then $\a(A \times B) \leq \frac{\a(B)}{2}$;

\item[(ii)] if $(G_i)_{i=1}^n$ are finite groups, then $\a\left(\prod_{i=1}^n G_i\right) < \frac{1}{n}$.
\end{itemize}
\end{lemma}
\begin{proof}
We have $\D_0(A \times B) = \min \{ \D_0(A), \D_0(B) \}$,
so that we can write
\[\begin{split}
	\a(A \times B)
	&= \frac{\log \left( \min \{\D_0(A),\D_0(B)\} \right)} {\log(|A||B|)}
	\\&= \frac{\log \left( \min \{|A|^{\a(A)},|B|^{\a(B)}\} \right)} {\log|A|+\log|B|}
	\\&\leq \a(B) \cdot \frac{\min\{\log|A|,\log|B|\}} {\log|A| + \log|B|} .
\end{split} \, \]
Point (i) follows.
For (ii), observe that there is a group, say $G_n$, satisfying $\log |G_n| \leq \frac{1}{n} \sum_{i=1}^n \log |G_i|$.
It follows from the previous computation that
$$ \a\left( \prod_{i=1}^n G_i \right) <
 1\cdot \frac{\min\{\sum_{i=1}^{n-1} \log|G_i|, \log|G_n|\}} {\sum_{i=1}^n \log|G_n|} \leq \frac{1}{n} . $$
 The proof is complete.
\end{proof}

\subsection{Quasisimple groups} \label{sec:extensions}

We recall that a {\itshape central extension} of a group $Q$ is a short exact sequence
$$ 1 \> \to \> N \> \hookrightarrow \> G \> \mapsto \> Q \> \to \> 1 $$
where $N \subseteq \Z(G)$.
If $G$ is perfect, then $G$ is a {\itshape perfect central extension}.
A finite group is {\itshape quasisimple} if it is a perfect central extension of a nonabelian finite simple group,
that is, $G$ is perfect and $G/\Z(G)$ is simple.

It is a key property of perfect groups that, if $G$ is perfect,
then there is a unique universal perfect central extension of $G$, which we denote by $\widetilde{G}$.
This extension is obtained by choosing $N$ isomorphic to the {\itshape Schur multiplier} of $G$,
that is, the second homology group $H^2(G,\mathbb{Z})$. For more details on this, we refer to \cite[Sec. 33]{Asc00}.

The next two lemmas are essentially well known.
We write down the proof of Lemma \ref{lemExtQuasi} below,
which is a little more general, and will be useful later.
 
\begin{lemma} \label{lemNQuasi} 
Let $G$ be a quasisimple group, and let $N \lhd G$ be a proper normal subgroup.
Then $N \subseteq \Z(G)$, and $G/N$ is again quasisimple.
Moreover, a perfect central extension of a quasisimple group is quasisimple.
\end{lemma}
\begin{proof} 
	See \cite[(31.2) and (33.5)]{Asc00}.
\end{proof} 

\begin{lemma} \label{lemExtQuasi}
Let $G$ be a finite perfect group.
Let $N \subseteq \Z(G)$, and $N \lhd N_0 \lhd G$ such that $N_0/N \subseteq \Z(G/N)$.
Then $N_0 \subseteq \Z(G)$.
\end{lemma}
\begin{proof}
Let $Z=\Z(G)$.
We first show that $\Z(G/N) = Z/N$. It is clear that $Z/N \subseteq \Z(G/N)$.
On the other hand, $\Z(G/N)$ is certainly contained in the projection of $\Z(G/Z)$ in $G/N$.
By Gr\"{u}n's Lemma \cite{Gru36} we have $\Z(G/Z)=1$, so $\Z(G/N) \subseteq Z/N$.
Now from the hypotheses we have $N_0/N \subseteq Z/N$, i.e. $N_0 \subseteq Z$.
\end{proof}

We conclude this subsection with the following observation,
which is useful in an inductive argument. 

\begin{lemma} \label{lemKey}
	Let $G$ be a nonquasisimple group in which all the proper quotients are quasisimple.
	Then either $G$ has a unique minimal normal subgroup, or $G$ is the direct product of two nonabelian finite simple groups.
\end{lemma}
\begin{proof}
	We first notice that either $G$ is a group of prime order, or it is perfect.
	In the first case the statement is trivial, thus we can suppose that $G$ is perfect.
	
	Let us suppose that $N_1$ and $N_2$ are two distinct minimal normal subgroups of $G$, so that $N_1 \cap N_2 =1$.
	Let $H=N_1N_2$, and note that $G/N_i$ is quasisimple for $i\in \{1,2\}$.
	If $G=H$, then $G$ is the direct product of finite simple factors.
	In particular, it is easy to see that these factors must be precisely two and nonabelian.
	
	Aiming for a contradiction, suppose $H \neq G$.
	Since all the proper normal subgroups of a quasisimple group are central,
	$H/N_i \subseteq \Z(G/N_i)$ for both $i\in \{1,2\}$.
	Let $h \in H$ and $g \in G$. For each $i\in \{1,2\}$, we have
	$$ [h,g]N_i = [hN_i,gN_i] \subseteq N_i . $$
	It follows that $[h,g] \in N_1 \cap N_2 =1$, and since $h \in H$ is arbitrary, we conclude $H \subseteq\Z(G)$.
	Finally, by Lemma \ref{lemNQuasi}, $G$ itself is quasisimple, against our initial assumption.
\end{proof}

\subsection{Linear and projective representations} 

So far, we have only dealt with complex linear representations, but modular and projective representations
will be a key tool in the proof of Theorem \ref{thm:B}.
For a prime $r \geq 2$,
we use $\D_r(G)$ to denote the minimal degree of a nontrivial irreducible modular representation of $G$ over a field of characteristic $r$.
We recall that a {\itshape projective representation} of $G$ is a representation of a perfect central extension of $G$.
Equivalently, it is a group homomorphism
$$ \rho \colon \> G \to \mathrm{PGL}(V) . $$
The degree of the representation is the dimension of $V$.
As for linear representations, we denote by $\Di_0(G)$ the minimal degree of a nontrivial complex projective representation,
and by $\Di_r(G)$ the minimal degree of a nontrivial projective representation over a field of characteristic $r$.
We remark that, if $G$ is finite and perfect, and $\widetilde{G}$ is its universal perfect central extension, then
$$ \Di_0(G) = \D_0(\widetilde{G}) \leq \D_0(G), \quad\hbox{and} \quad \Di_r(G) =\D_r(\widetilde{G}) \leq \D_r(G) . $$
In the remainder of this subsection, we focus on the projective representations of finite simple groups of Lie type.
We recall that, if $G$ is a group of Lie type over a field of characteristic $p$,
then a modular representation of $G$ over a vector space $V$ is in {\itshape natural characteristic} when the characteristic of $V$ is $p$,
while it is in {\itshape cross-characteristic} otherwise.

\begin{lemma} \label{lemMinimalV}
	Let $L$ be a finite simple group of Lie type over $\mathbb{F}_q$, $q=p^f$.
	If $V$ is a vector space in natural characteristic $p$ on which $\widetilde{L}$ acts nontrivially, then
	$$ |V| \> \geq \> q^{\, \Di_p(L)} . $$
\end{lemma} 
\begin{proof} 
It follows immediately from the definition.
\end{proof}

For the proof of Theorem \ref{thm:B}, we need detailed information about the $\frac{1}{4}$-quasirandom finite simple groups.
The infinite families of $\frac{1}{4}$-quasirandom simple groups are
	\begin{gather*}
	\PSL_d(q) \hbox{ with }d\in\{2,3\}, \quad
	\mathrm{Sp}_4(q) \hbox{ with } q \hbox{ even}, \\
	G_2(q) \hbox{ with } q=3 \> (\mathrm{mod}\,6), \quad
	^2B_2(q^2), \quad
	^2G_2(q^2) \,. \end{gather*} 
	The relevant information can be recovered from \cite{LS74,SZ93} for classical groups,
	and from \cite[Sec. 4]{Lub01} for exceptional groups (see also the Appendix).
	We collect these data in Table \ref{table:1/4}.
	
{\small 
\begin{table}[ht] 
	\centering
	\begin{tabular}{| c c | c c c c |} 
		\hline
		
		\phantom{$\dfrac{1^1}{1_{1_1}}$}\hspace{-6mm}
		$L$ & Comments  & $\D_0(L)$ & $\Di_0(L)$ & $\Di_r(L)$, $r\ne p$ & $\Di_p(L)$ \\
		
		\hline \hline
		
		\phantom{$\dfrac{1^1}{1_{1_1}}$}\hspace{-6mm}
		$\PSL_2(q)$ & $q$ even & $q-1$ & $q-1$ & $q-1$ & $2$ \\
		
		\phantom{$\dfrac{1^1}{1_{1_1}}$}\hspace{-6mm}
		& $q = 1\, (\mathrm{mod}\,4) $ & $\dfrac{q+1}{2}$ & $\dfrac{q+1}{2}$ & $\dfrac{q-1}{2}$ & $2$ \\
		
		\phantom{$\dfrac{1^1}{1_{1_1}}$}\hspace{-6mm}
		& $q = 3\, (\mathrm{mod}\,4) $ & $\dfrac{q-1}{2}$ & $\dfrac{q-1}{2}$ & $\dfrac{q-1}{2}$ & $2$ \\
		
		\phantom{$\dfrac{1^1}{1_{1_1}}$}\hspace{-6mm}
		$\PSL_3(q)$ & $q \ne 2,4$ & $q(q+1)$ & $q(q+1)$ & $q^2+q-1$ & $3$ \\
		
		\phantom{$\dfrac{1^1}{1_{1_1}}$}\hspace{-6mm}
		$\mathrm{Sp}_4(q)$ & $q$ even & $\dfrac{1}{2}q(q-1)^2$ & $\dfrac{1}{2}q(q-1)^2$ & $\dfrac{1}{2}q(q-1)^2$ & $4$ \\
		
		\phantom{$\dfrac{1^1}{1_{1_1}}$}\hspace{-6mm}
		$G_2(q)$ & $q=3 (\mathrm{mod}\,6)$ & $\phi_3(q)\phi_6(q)$ & $\phi_3(q)\phi_6(q)$ & $\phi_3(q)\phi_6(q) - \delta_1$ & $7$ \\
		
		\phantom{$\dfrac{1^1}{1_{1_1}}$}\hspace{-6mm}
		$^2B_2(q^2)$ & $q=\sqrt{2^{2m+1}}$ & $\dfrac{\sqrt{2}}{2} q \phi_1(q) \phi_2(q)$ & $\dfrac{\sqrt{2}}{2} q \phi_1(q) \phi_2(q)$ & $\dfrac{\sqrt{2}}{2} q \phi_1(q) \phi_2(q)$ & $4$ \\
		
		\phantom{$\dfrac{1^1}{1_{1_1}}$}\hspace{-6mm}
		$^2G_2(q^2)$ & $q=\sqrt{3^{2m+1}}$ & $\phi_{12}(q)$ & $\phi_{12}(q)$ & $\phi_{12}(q)-\delta_2$ & $7$ \\
		
		\hline
		
	\end{tabular}
	\medskip
	{\caption{The $\frac{1}{4}$-quasirandom simple groups of Lie type.} \label{table:1/4}}
\end{table}
}

    We denote by $\phi_\ell(q)$ the $\ell$-th cyclotomic polynomial in the variable $q$.
	The value of $\delta_1$ is $1$ if $r=2$, $q=3^f$, $f\geq 2$, or $r=3$, $3$ divides $q-1$, and it is $0$ otherwise.
	The value of $\delta_2$ is $1$ if $r=2$, and $0$ otherwise.
	
	\begin{remark} \label{rem14simple} 
	We observe that, if $L$ is a $\frac{1}{4}$-quasirandom simple group of Lie type, in cross-characteristic we have
	$$ \Di_r(L) \> \geq \> \D_0(L) -1 . $$
	Moreover, we point out that the comparison between $\Di_0(L)$ and $\Di_r(L)$ ($r \neq p$) can be explicitly found,
	for all simple groups of Lie type, in \cite[Sec. 3]{Tie03}.
\end{remark}

\section{Quasirandom groups of affine type} \label{sec3} 

In this short section, we provide some highly quasirandom nonquasisimple groups.
The key tool is Lemma \ref{lemAffSp} below.
Before giving the results, we revise some notions from the representation theory of finite groups.
Let $V$ be a vector space over $\mathbb{C}$, and let $V^*$ be its dual space.
We remark that $V^* \cong \mathrm{Hom}(V,\mathbb{C}^*)$.
If $\phi \colon V \rightarrow V$ is a homomorphism of vector spaces, then $\phi^* \colon V^* \rightarrow V^*$ is defined by
$$ \phi^* (f) \> := \> f \circ \phi . $$
Let $K$ be a finite group.
If $\rho \colon K \rightarrow \GL(V)$ is a representation of $K$,
then the dual representation $\rho^* \colon K \rightarrow \GL(V^*)$ is defined by
$$ \rho^*(g) \> := \> \rho(g^{-1})^* $$
for every $g \in K$.
It is well known that the representations $\rho$ and $\rho^*$ are not isomorphic in general.

\begin{lemma} \label{lemAffKey} 
	Let $V$ be a finite vector space, and $K \leqslant \GL(V)$.
	Then the number of orbits of $K$ on $V$ equals the number of orbits of $K$ on $V^*$.
\end{lemma}
\begin{proof}
Let $V=(\mathbb{F}_q)^n$ and $g \in K$.
The number of $g$-fixed points in $V$ is $q^{\dim(Ker(g-I))}$,
while the number of $g$-fixed points in $V^*$ is $q^{\dim(Ker((g^{-1})^T-I))}$,
where $g^T$ is the transpose of $g$.
Now $\dim(Ker(g-I))$ is the multiplicity of $1$ as an eigenvalue of $g$.
We remark that the eigenvalues of $(g^{-1})^T$ are the eigenvalues of $g^{-1}$,
which are the inverses of the eigenvalues of $g$.
In particular, the two quantities above are equal.
Since $g$ is arbitrary, the proof follows from Burnside's lemma,
i.e. the number of orbits is equal to the average number of fixed points.
\end{proof} 

\begin{lemma} \label{lemAffSp} 
	Let $G = V \rtimes K$ be a $2$-transitive affine group.
	Then
	$$ \D_0(G) \> = \> \D_0(K) . $$
\end{lemma}
\begin{proof}
	Since $K$ is a quotient of $G$,
	it is sufficient to show that $\D_0(G) \geq \D_0(K)$.
	We first notice that $K$ acts nontrivially on $V \setminus \{0\}$,
	so $\D_0(K) \leq |V|-2$ from Lemma \ref{lemActions}.
	Moreover, we can identify $V^*$ with $\Irr(V)$.
	Let $\pi \in \Irr(G) \setminus \{1\}$.
	If $\pi_{\downarrow V}$ is the trivial representation of $V$, then $\pi \in \Irr(G/V)$ is an irreducible representation of $K$.
	In particular, $\dim(\pi) \geq \D_0(K)$.
	On the other hand, if $\pi_{\downarrow V}$ is a nontrivial representation of $V$,
	then, by Clifford theory, there is an irreducible character $\theta \in \Irr(V) \setminus \{1\}$ such that
	$\dim(\pi)$ is a multiple of $|K:I_K(\theta)|$, where $I_K(\theta)$ is the inertia subgroup of $\theta$. 
	From Lemma \ref{lemAffKey}, $K$ acts transitively on $\Irr(V) \setminus \{1\}$.
	By the orbit-stabilizer theorem, we have
	$$ |K:I_K(\theta)| = |\theta^K| = |V|-1 . $$
	From the discussion above, we obtain
	$$ \dim(\pi) \geq |V|-1 > \D_0(K) . $$
	Since $\pi \in \Irr(G) \setminus \{1\}$ is arbitrary, we conclude that $\D_0(G) \geq \D_0(K)$ as desired.
\end{proof}

\begin{example} \label{ex:ASp} 
	Let $q \geq 4$ range among the powers of $2$, and let $G_q = (\mathbb{F}_q)^4 \rtimes \Sp_4(q)$. 
	Since $\Sp_4(q)$ is transitive on $(\mathbb{F}_q)^4 \setminus \{0\}$, from Lemma \ref{lemAffSp} we have
	\[\begin{split}
		\a(G_q) &= \frac{\log \D_0 ( \Sp_4(q) )}{\log |G_q|}
		\\&= \frac{\log \left( \dfrac{q(q-1)^2}{2} \right) }{\log (q^8(q^4-1)(q^2-1))} .
	\end{split} \, \]
	A direct computation shows that this value is more than $\frac{1}{5}$ for $q \geq 64$,
	and it lies between $\frac{3}{14} - \frac{1}{\log_2 q}$ and $\frac{3}{14}$,
	for all $q$. In particular, $\lim_{q \rightarrow +\infty} \a(G_q) = \left( \frac{3}{14} \right)^-$.
	Similarly, if $q$ is a prime power and $H_q = (\mathbb{F}_q)^2 \rtimes \SL_2(q)$,
	then from Lemma \ref{lemAffSp} we have
	$$ \a(H_q) = \frac{\log \D_0(\SL_2(q))}{\log|H_q|} . $$
	A direct computation shows that $\a(H_q)$ lies between $5^{-1} -q^{-1}$ and $5^{-1}$ for all $q$,
	so in particular $\lim_{q \rightarrow +\infty} \a(H_q) = \left( \frac{1}{5} \right)^-$.
\end{example}

\section{Proof of the main theorems} \label{sec4} 

\subsection{Proof of Theorem \ref{thm:A}} 

We immediately prove the second part of the theorem.
We restate it as follows.

\begin{lemma} \label{lemNew}
	Fix $\varepsilon,\delta >0$, and let $G$ be a perfect group.
	Suppose that $G/\Sol(G)$ is $\varepsilon$-quasirandom, and that $|G/\Sol(G)| \geq |G|^\delta$.
	If $|G|$ is sufficiently large with respect to $\varepsilon$ and $\delta$,
	then $G$ is $\frac{\varepsilon \delta}{3}$-quasirandom.
\end{lemma}
\begin{proof}
By \cite[Lemma 4.10]{EMPS23}, we have
	$$ \D_0(G) \> \geq \> c \, \D_0(G/\Sol(G))^{1/2} , $$
	where $c$ is an absolute constant.
It follows that
$$ |G|^{\a(G)} \> \geq \> c |G/\Sol(G)|^{\a(G/\Sol(G))/2} , $$
and passing to the logarithms,
	$$ \a(G) \> \geq \> \frac{\a(G/\Sol(G))}{2} \cdot \frac{\log|G/\Sol(G)|}{\log|G|} - \frac{\log (c^{-1})}{\log|G|}  . $$
The proof follows.
\end{proof} 

The main idea behind the other direction of Theorem \ref{thm:A} is the following.
Under the hypothesis that $\Sol(G)=1$,
we prove that $G=F^*(G)$, the generalized Fitting subgroup of $G$.
If this is not the case, then $G/F^*(G)$ has a nontrivial simple quotient $L$.
We use the quasirandomness of $L$ to obtain an upper bound on $|G|$,
while Bender's theorem and other considerations provide a lower bound on $|F^*(G)|$.
These inequalities fight each other, giving the desired contradiction.
A similar argument will be used in a portion of the proof of Theorem \ref{thm:B}.
Going into details, we start recording the Schreier conjecture,
which is now a theorem as a consequence of the Classification of the Finite Simple Groups.

\begin{theorem}[Schreier conjecture] \label{thSchreier} 
	If $L$ is a finite simple group, then the group of the outer automorphisms $\Out(L)$ is solvable.
\end{theorem} 

The following two observations concern sections of finite groups.

\begin{remark} \label{remSimSec}
		Let $N \lhd G$, and suppose that a section $H_0/N_0 \cong L$ of $G$ is simple.
		Then $N \subseteq N_0$, or  $NN_0 \geqslant H_0$.
		In the first case $L$ appears as a section of $G/N$, while in the second case $L$ appears as a section of $N$.
	\end{remark} 
	
	\begin{lemma} \label{lemSymSec}
	Let $\Omega$ be a finite set, and let $N \lhd H \leqslant \Sym(\Omega)$.
	If $H/N \cong L$ is simple, then $\Sym(\Omega)$ has a subgroup isomorphic to $L$.
\end{lemma}
\begin{proof}
Let $|\Omega|=n$, $x \in \Omega$, and work by induction on $n$.
Let $H_x$ be stabilizer of $x$.
We claim that $NH_x/N$ is a corefree subgroup of $H/N$ of index at most $n$.
First,
$$ |(H/N):(NH_x/N)| = |H:NH_x| \leq |H:H_x| \leq n . $$
Now, if $NH_x$ is neither $N$ nor $H$, then the proof is finished by the fact that $H/N$ is simple.
In the first case $|H/N| \leq n$, and the proof follows by the Cayley embedding theorem.
If $NH_x=H$, then $H_x/H_x \cap N \cong H/N \cong L$,
hence by induction $\Sym(\Omega \setminus \{x\})$ and so $\Sym(\Omega)$ has a subgroup isomorphic to $L$, as desired.
\end{proof}

In the following remark, we discuss the asymptotic behaviour of $\D_0(L)$ and $\a(L)$,
when $L$ is a finite simple group, and $|L| \rightarrow +\infty$.
Of course, we do not care about the sporadic simple groups in this context.

\begin{remark} \label{rem:LieCon} 
For $n \geq 6$,
$$ \D_0(\Alt(n)) = n-1, \quad \hbox{and} \quad \a(\Alt(n)) = \frac{\log (n-1)}{\log (n!/2)} . $$
In particular, for $L=\Alt(n)$, we have
$$ \D_0(L) \sim \frac{\log|L|}{\log\log|L|} , \quad \hbox{and} \quad \a(L) \sim \frac{\log\log|L|}{\log|L|} $$
when $|L| \rightarrow +\infty$.
These are the less quasirandom among the nonabelian finite simple groups.
If $L=L_d(q)$ is a finite simple group of Lie type, then it is easy to check from the Appendix that
	$$ q^{10^{-1}d} \> \leq \> \D_0(L_d(q)) \> \leq \> q^{10d} , $$
	and
	$$ 10^{-1} d^{-1} \> \leq \> \a(L_d(q)) \> \leq \> 10d^{-1} . $$
	In particular, if $L=L_d(q)$ and $q$ is fixed, then $\D_0(L)$ is comparable with $c^{\sqrt{\log|L|}}$,
	while $\a(L)$ is comparable with $\frac{c}{\sqrt{\log|L|}}$.
\end{remark} 

\begin{proof}[Proof of Theorem \ref{thm:A}]
After Lemma \ref{lemNew}, we have only to prove the first part of the theorem.
Fix $\varepsilon>0$.
From Corollary \ref{corActions}(i), it is clear that $|G/\Sol(G)| \geq |G|^\varepsilon$.
Moreover, $G/\Sol(G)$ is again $\varepsilon$-quasirandom.
Up to replacing $G$ with $G/\Sol(G)$, we can assume that $G$ has a trivial solvable radical.
Now it is well known that the generalized Fitting subgroup $F=F^*(G)$ coincides with the socle of $G$,
and it is the direct product of nonabelian simple groups $T_i$'s.
Let $F= \prod_i (T_i)^{n_i}$, where $n_i \geq 1$ for each $i$.
Moreover, as $\C_G(F) \subseteq F$ \cite[(31.13)]{Asc00}, we have $\C_G(F) = \Z(F) =1$.
It follows that $G$ can be embedded as a subgroup of
$$ \Aut(F) \> \cong \> \prod_i \Aut(T_i)^{n_i} \rtimes \Sym(n_i) . $$
In particular, all the composition factors of $G/F$ lie as sections of one among $\Out(T_i)$ or $\Sym(n_i)$, for some $i$.
We claim that $G=F$.
Otherwise, it is easy to see that $G/F$ has a simple quotient $L$ which is $\varepsilon$-quasirandom.
 By the definition of quasirandom degree, we have
\begin{equation} \label{eqThA}
|G| \> \leq \> \D_0(L)^{1/\varepsilon} .
\end{equation}
So, if $\D_0(L)$ is bounded, then $|G|$ is bounded.
Let $n = \max_i n_i$, so that $|F| \geq 60^n$.
From Remark \ref{remSimSec} and Theorem \ref{thSchreier}, $L$ has to appear as a section of $\Alt(n)$.
From Lemma \ref{lemSymSec}, $\Sym(n)$ has a subgroup isomorphic to $L$.
From Lemma \ref{lemActions} we obtain $\D_0(L) \leq n$, and so
$$ |F| \> \geq \> 60^{\D_0(L)} . $$
When $\D_0(L)$ is sufficiently large, this gives a contradiction to (\ref{eqThA}).
We have proven that $G/\Sol(G)=F^*(G/\Sol(G))$, so $G/\Sol(G)$ is an $\varepsilon$-quasirandom direct product of finite simple groups,
say $G/\Sol(G) = \prod_i (T_i)^{n_i}$.
From Remark \ref{rem:LieCon}, each $T_i$ is a simple group of Lie type of $\varepsilon$-bounded rank.
Let $N= \sum_i n_i$.
From Lemma \ref{lemDP}(ii), we have $\varepsilon \leq \a(G/\Sol(G)) \leq N^{-1}$,
so that $N \leq \varepsilon^{-1}$.
\end{proof} 

\begin{remark} \label{remAtt} 
An attentive reader will have noticed that a bound of the type $\D_0(G) \geq c \log|G|$ is strong enough to give that
$G/\Sol(G)$ is the direct product of nonabelian finite simple groups.
\end{remark}

It is natural to ask whether the condition $|G/\Sol(G)| \geq |G|^\varepsilon$ is actually superflous in the second part of Theorem \ref{thm:A},
as it may be implied by the quasirandomness of the sole $G/\Sol(G)$.
The answer is no, as it can even happen that $|G/\Sol(G)|$ is bounded:
it is enough to consider the perfect deleted permutation module $(C_n)^4 \rtimes \Alt(5)$ when $n$ is large.
The following example, where $|G/\Sol(G)|$ is unbounded but too small, is due to S. Eberhard.

\begin{example}
	Let $p \geq 5$ be a prime, and $d \geq 2$.
	We define $G_d$ as the group of the $2d \times 2d$ upper triangular matrices whose diagonal entries
	are identical matrices in $\SL_2(p)$, and the upper entries are arbitrary $2 \times 2$ matrices over $\mathbb{F}_p$.
	For instance,
	$$
	G_3 \> = \>
	\left(
	\begin{array}{ccc}
		g & * & *  \\
		0 & g & *  \\
		0 & 0 & g  \\
	\end{array}
	\right) ,
	$$
	where $g \in \SL_2(p)$, and $*$ is arbitrary.
	We have
	$$ |G_d| = \frac{1}{2}(p^3-p) p^{2d(d-1)} \geq p^{d^2} . $$
	It can be proven that $G_d$ is perfect, and its solvable radical is the subgroup where $g \in \Z(\SL_2(p))$ is a scalar matrix.
	In particular, $G_d/\Sol(G_d)$ is isomorphic to $\PSL_2(p)$, which is $0.34$-quasirandom for large $p$, but its cardinality is less than $p^3$.
\end{example}

\subsection{Proof of Theorem \ref{thm:B}} 

We first prove the shallow version of Theorem \ref{thm:B}.
This is a nice application of (\ref{eqN}).

\begin{proof}[Proof of Corollary \ref{cor:1/3}]
From the data collected in the Appendix,
we note that, among the finite simple groups, $J_1$ and $O'N$ are the only ones whose quasirandom degree is at least $\frac{1}{3}$.
Suppose that $G$ is of minimal cardinality among the finite groups satisfying
$\a(G) \geq \frac{1}{3}$ that are not isomorphic to either $J_1$ or $O'N$.
Of course, there exists a nontrivial proper normal subgroup $N$ of $G$.
We have $\a(G/N) \geq \frac{1}{3}$, and $G/N$ is isomorphic to either $J_1$ or $O'N$.
In the former case, from (\ref{eqN}), we have
$$ |N| \leq \frac{\D_0(J_1)^{1/\a(G)}}{|J_1|} \leq \frac{56^3}{175\,560} \approx 1.00031 . $$
So $N=1$, a contradiction.
Now suppose $G/N \cong O'N$. From (\ref{eqN}),
$$ |N| \leq \frac{\D_0(O'N)^{1/\a(G)}}{|O'N'|} \leq \frac{10 \, 944^3}{460\,815\,505\,920} \approx 2.84446 , $$
so that $N$ is the cyclic group of order $2$. 
Since the only nonidentity element of $N$ has to be fixed by conjugation, $N$ is central in $G$. In particular, $G$ is a perfect central extension of $O'N$ by $N$. On the other hand, the Schur multiplier of $O'N$ has cardinality $3$, and such an extension does not exist.
\end{proof} 

The following observation is not really needed in the proof of Theorem \ref{thm:B},
but it makes the structure of $G$ more transparent.

\begin{lemma} \label{lemUniqueM}
Suppose that $\a(G) \geq \frac{1}{5}$, or that $\a(G) \geq \frac{1}{6}$ and $|G| \geq 2 \cdot 10^{24}$.
Then $G$ has a unique maximal normal subgroup.
\end{lemma}
\begin{proof}
From Corollary \ref{cor:1/3}, for every finite group $G$ we have
$$ \a(G) \leq \max \left\{ \a(J_1) , \a(O'N) \right\} = \a(O'N) < \frac{2}{5} , $$
while $\a(G) <\frac{1}{3}$ if $G$ is not isomorphic to $J_1$ or $O'N$.
Suppose that $M_1$ and $M_2$ are two distinct maximal normal subgroups of the group $G$.
Since $M_1M_2/M_1$ is a normal subgroup of the simple group $G/M_1$, we obtain that $G=M_1M_2$.
So $M_1/(M_1 \cap M_2)$ and $M_2/(M_1 \cap M_2)$ are normal subgroups of $G/(M_1 \cap M_2)$ with trivial intersection.
It follows that
$$ G/(M_1\cap M_2) = M_1/(M_1\cap M_2) \times M_2/(M_1\cap M_2) \cong G/M_2 \times G/M_1 . $$
Lemma \ref{lemDP}(i) yields
$$ \a(G) \leq \a( (G/M_2) \times (G/M_1) ) < \frac{1}{5} , $$
and $\a(G) < \frac{1}{6}$ if neither $J_1$ or $O'N$ appears as a quotient of $G$.
If one among $J_1$ or $O'N$ appears as a quotient, then
$$ |G| \leq \D_0(G)^6 < \D_0(O'N)^6 = 10\,944^6 < 2 \cdot 10^{24} , $$
because $\D_0(J_1) < \D_0(O'N)$.
\end{proof}

We are ready to prove Theorem \ref{thm:B}.
This proof has a number of cases, the most difficult, which we treat in the end, being the following.
We have an abelian minimal normal subgroup $N$ of $G$ such that $G/N$ is quasisimple,
and the action of $G/N$ on $N$ by conjugation provides a
	projective modular representation of the simple quotient of $G/N$, say $L$.
	The key point is that the (complex, linear) quasirandomness of $L$ controls also the degrees of the modular projective representations.
	This gives a lower bound on $|N|$, which we put against the upper bound given by (\ref{eqN}).
	We start with the details.

\begin{remark} 	\label{rem:calculator}
The veracity of Theorem \ref{thm:B} has been checked with a computer for all the finite groups of order up to $2\cdot 10^6$,
using the library \texttt{PerfectGroups} in the GAP System \cite{GAP}.
\end{remark} 

	Let $G$ be a counterexample to Theorem \ref{thm:B} whose order is minimal.
	In particular, $\a(G) \geq \frac{1}{5}$ and $G$ is not quasisimple.
	By Lemma \ref{lemUniqueM}, $G$ has a unique maximal normal subgroup and cannot be written as a direct product of nontrivial groups.
	Therefore, using also Lemma \ref{lemKey}, we find a chain
	$$ 1 \> \lhd \> N \> \lhd \> M \> \lhd \> G  $$
	where $N$ is the unique minimal normal subgroup of $G$, and $M$ is the unique maximal normal subgroup of $G$.
	It follows that $G/M$ is isomorphic to a simple group $L$ satisfying $\a(L) > \frac{1}{5}$,
	and that $N \cong T^n$ is the direct product of isomorphic simple groups. 
	We distinguish two cases: either $G/N$ is quasisimple or not.
	Let us start from the latter.
	
	\medskip \noindent \textsc{\underline{$G/N$ is not quasisimple.}}
	We have $\a(G/N) > \frac{1}{5}$, and, by the minimality of $G$, we are in one of the two cases:
	\begin{enumerate}[$(a)$]
		\item $G/N \cong (\mathbb{F}_3)^3 \rtimes \SL_3(3)$;
		\item $G/N$ is a perfect central extension of $(\mathbb{F}_{q})^4 \rtimes \Sp_4(q)$, $q=2^f$, $f \geq 6$.
	\end{enumerate}

	If we are in case $(a)$, using (\ref{eqN}) we obtain
	$$ |N| \leq \frac{\D_0((\mathbb{F}_3)^3 \rtimes \SL_3(3))^5}{|(\mathbb{F}_3)^3 \rtimes \SL_3(3)|} = \frac{12^5}{151\,632} \approx 1,64 . $$
	Hence $N$ is trivial, a contradiction.
	
	If we are in case $(b)$, we remark that $G/M \cong \Sp_4(q)$ and $\a(G/N) <\frac{3}{14}$.
	Let $N \lhd N_0 \lhd M$ such that $G/N_0 \cong (\mathbb{F}_{q})^4 \rtimes \Sp_4(q)$, $N_0/N \subseteq \Z(G/N)$.
	First suppose that $N$ is abelian.
	If $N \subseteq \Z(G)$, then $N_0 \subseteq \Z(G)$ by Lemma \ref{lemExtQuasi}.
	Thus $G$ is a perfect central extension of $(\mathbb{F}_q)^4 \rtimes \Sp_4(q)$, against the assumptions.
	 If $N \nsubseteq \Z(G)$, then by Lemma \ref{lemImprove},
	$$ \frac{3}{14} > \a(G/N) > \frac{\a(G)}{1-\a(G)} \geq \frac{1}{4} . $$
	This is impossible.
	It remains the case where $N$ is nonabelian, say $N= T^n$ where $T$ is a nonabelian finite simple group.
	We argue in a similar way as in Theorem \ref{thm:A}:
	since $N$ is the unique minimal normal subgroup of $G$,
	we have $\C_G(N) =1$ and the same argument works.
	We have
	\begin{equation} \label{eqSp}
	|G| \> \leq \> \D_0( \Sp_4(q) )^5 . 
	\end{equation}	 
	Considering the embedding of $G$ into $\Aut(T^n)$,
	and arguing as before, we obtain $\D_0( \Sp_4(q) ) \leq n$.
	Since $|N| \geq 60^n$, and $60^x > x^5$ for every $x \geq 1$, this contradicts (\ref{eqSp}).
	
	\medskip \noindent \textsc{\underline{$G/N$ is quasisimple.}}
	Let $C=\C_G(N)$. We distinguish three cases.
	
	\smallskip \noindent \textsc{Suppose that $C=G$.}
	Here $N$ is central and $G$ is quasisimple by Lemma \ref{lemNQuasi}, a contradiction.
	
	\smallskip \noindent \textsc{Suppose that $C=1$.}
	Here $N=T^n$ is the product of nonabelian finite simple groups, and $G$ is embedded into $\Aut(T^n)$.
	Let $G/M=L$. Arguing again as in the proof of Theorem \ref{thm:A}, we obtain
	$$ |G| \leq \D_0(L)^5, \hspace{0.5cm} \mbox{and} \hspace{0.5cm} |N| \geq 60^{\D_0(L)} . $$
	This is impossible for any value of $\D_0(L)$.
	
	\smallskip \noindent \textsc{Suppose that $N \subseteq C \subseteq M$.}
	This case will take up the rest of the proof.
	First, $N$ is an elementary abelian noncentral subgroup, and Lemma \ref{lemImprove} provides
	$$ \a(L) \geq \a(G/N) > \frac{1}{4} . $$
	Moreover, the action of $G/N$ on $N$ by conjugation provides a
	projective modular representation of $L$ (recall that $G/N$ is quasisimple).
	As already explained, the key idea is that the quasirandomness of $G/N$ gives a lower bound on $|N|$,
	which we put against the upper bound given by (\ref{eqN}).
	We observe that $\Alt(5)$ and $\Alt(6)$ are $\frac{1}{4}$-quasirandom, but if one among them appears as $L$,
	then
	$$ |G| \leq \D_0(L)^5 \leq 5^5 = 3\,125 , $$
	and we can easily work wit a computer (see Remark \ref{rem:calculator}).
	One last time, we divide the discussion into three cases:
	\begin{enumerate}[$(i)$]
		\item $L$ is a simple group of Lie type
		over a field of characteristic $p$, and $G/N$ induces a projective modular representation of $L$ in cross-characteristic;
		\item $L$ is a simple group of Lie type
		over a field of characteristic $p$, and $G/N$ induces a projective modular representation of $L$ in natural characteristic;
		\item $L$ is a sporadic simple group.
	\end{enumerate}
	
	\smallskip \noindent \textsc{Let $L$ be as described in $(i)$.}
	We remark that $|L| > \D_0(L)^3$ by Corollary \ref{cor:1/3}.
	Hence, by (\ref{eqN}), 
	$$ |N| \leq |M| \leq \frac{\D_0(L)^5}{|L|} < \D_0(L)^2 . $$
	From Remark \ref{rem14simple}, and the definition of $\Di_r(L)$ with $r\ne p$ a prime, we obtain
	$$ 2^{\D_0(L) -1} \leq 2^{\Di_r(L)} \leq |N| < \D_0(L)^2  . $$
	The above inequality is possible only if $\D_0(L) \leq 6$, which gives
	$$ |G| \leq \D_0(L)^5 \leq 6^5 = 7 \, 776 . $$
	This case can be handled computationally.
	
	\smallskip \noindent \textsc{Let $L$ be as described in $(ii)$.} 
	We first deal with $L \cong \mathrm{Sp}_4(q)$ with $q$ even, $q \geq 4$.
	Since the Schur multiplier is trivial, we have $N=M$.
	By (\ref{eqN}), we find
	\[\begin{split}
		|N|
		&\leq \frac{\D_0(\mathrm{Sp}_4(q))^5}{|\mathrm{Sp}_4(q)|} 
		\\&\leq  \frac{(q^2-1)^5(q^2-q)^5}{2^5(q+1)^{5}}\cdot \frac{1}{q^4(q^4-1)(q^2-1)}
		\\&< q^5 . 
	\end{split} \] 
	On the other hand, by Lemma \ref{lemMinimalV} and Table \ref{table:1/4},
	$$ |N| \geq q^4 . $$
	So $|N|=q^4$, and we are in the situation of Theorem \ref{thm:B}\ref{thm:B:b}.
	
	The argument to deal with the other cases is the same, using (\ref{eqN}), Lemma \ref{lemMinimalV} and Table \ref{table:1/4}.
	When $L \cong \PSL_2(q)$,
	$$ |N| \leq \frac{\D_0(\PSL_2(q))^5}{|\PSL_2(q)|} \leq \frac{(q-1)^5}{q(q^2-1)} < q^2 . $$
	On the other hand,
	$$ |N| \geq q^2 , $$
	which provides a contradiction.
	If $L \cong \PSL_3(q)$, $q \geq 4$, we obtain
	$$ q^3 \leq |N| \leq \frac{q^5(q+1)^5}{q^3(q^3-1)(q^2-1)(q-1)} < q^3 . $$
	When $L \cong G_2(q)$, $q=3(\mathrm{mod}\,6)$, we have
	$$ q^7 \leq |N| \leq  \frac{(q+1)^5(q^2-q+1)^5}{q^6(q^6-1)(q^2-1)} < q^7 . $$
	For $^2B_2(q^2)$, 
	$$ q^8 \leq |N| \leq \frac{q^5(q^2-1)^5}{2^{5/2}q^4 (q^4-1)(q^2-1)} < q^5 . $$
	Finally, for $^2G_2(q^2)$,
	$$  q^{14} \leq |N| \leq \frac{(q^4-q^2+1)^5}{q^6(q^6-1)(q^2-1)} < q^7 . $$
	In all these four cases, the chain of inequalities is impossible.
	
	\smallskip \noindent \textsc{Let $L$ be as described in $(iii)$.}
	Since $\a(L) > \frac{1}{4}$, the possibilities for $L$ are
	$$ M_{11}, \quad
	J_1, \quad
	J_3, \quad 
	^2F_2(4)', \quad
	O'N . $$ 
	Observe that $M_{11}$, $J_1$ and $^2F_2(4)'$ admit no central perfect extensions,
	while the perfect central extensions of $J_3$ and $O'N$ are not $\frac{1}{5}$-quasirandom.
	Therefore $N=M$ is the unique normal subgroup of $G$,
	and $N$ can be thought as the vector space of a nontrivial modular representation of $L$.
	For each sporadic group,
	we use \cite{Jan05} to obtain a lower bound on $|N|$,
	and we compare the result with (\ref{eqN}).
	For $L \cong M_{11}$, we have
	$$ 3^5 \leq |N| \leq \frac{10^5}{|M_{11}|} \leq 13 ; $$
	For $L \cong J_1$, 
	$$ 2^{20} \leq |N| \leq \frac{56^5}{|J_1|} \leq 3 \, 138 ; $$
	For $L \cong J_3$,
	$$ 3^{18} \leq |N| \leq \frac{85^5}{|J_3|} \leq 89 ; $$
	For $L \cong \> ^2F_2(4)'$,
	$$ 2^{26} \leq |N| \leq \frac{104^5}{|^2F_2(4)'|} \leq 678 ; $$
	Finally, for $L \cong O'N$,
	$$ 3^{154} \leq |N| \leq \frac{10\,944^5}{|O'N|} \leq 10^{10} . $$
	These contradictions conclude the proof of Theorem \ref{thm:B}.

\appendix 
\section*{Appendix}
\section{Appendix: Quasirandom degree of finite simple groups} \label{appendix} 

We already discussed the alternating groups in Remark \ref{rem:LieCon}.
In particular, we observed that they are the less quasirandom among the nonabelian finite simple groups.

In the following tables, we collect the values $\D_0(L)$ and $\a(L)$ when $L$ is a finite simple group of Lie type or a sporadic simple group.
For a group of Lie type $L_d(q)$, we fix $d$ and make $q$ go to infinity.
The case where $L$ is of exceptional type is discussed in detail by L\"ubeck \cite{Lub01}.
So let $L$ be a classical group of Lie type.
In \cite{TZ96}, Tiep and Zalesskii computed the minimal dimension of a projective complex irreducible representation of $L$,
which we denote by $\Di_0(L)$. See also \cite{MM21}.
 In Table \ref{table:ClassicOrdinary}, we use this information to record $\D_0(L)$ for the simple classical groups.
 We remark that $\Di_0(L) < \D_0(L)$ holds in some cases, for example for some symplectic groups in odd characteristic.
 
 In Tables \ref{tableLieCl}, \ref{tableLieEx}, and \ref{table:Spor},
we collect the asymptotic quasirandom degrees of the finite simple groups of Lie type,
and the quasirandom degrees of the sporadic simple groups \cite{ATLAS}.
We remark that, after $(\PSL_2(q))_{q \geq 4}$, which is asymptotically $\frac{1}{3}$-quasirandom,
the most quasirandom families of finite simple groups are
$(\mathrm{Sp}_4(q))_{q =2^f}$ and the Suzuki groups $(^2B_2(q^2))_{q \geq 8}$,
whose quasirandom degrees are asymptotically $\frac{3}{10}$.

{\small 
\begin{table}[ht]
	\centering
	\begin{tabular}{| c c | c |}
		\hline
		
		\phantom{$\dfrac{1^1}{1_{1_1}}$}\hspace{-6mm}
		$L$ & Comments  & $\D_0(L)$ \\
		
		\hline \hline
		
		\phantom{$\dfrac{1^1}{1_{1_1}}$}\hspace{-6mm}
		$\PSL_d(q)$ & $d=2$, $q$ even & $q-1$ \\
		
		\phantom{$\dfrac{1^1}{1_{1_1}}$}\hspace{-6mm}
		& $d=2$, $q = 1\, (\mathrm{mod}\,4) $ & $\dfrac{q+1}{2}$ \\
		
		\phantom{$\dfrac{1^1}{1_{1_1}}$}\hspace{-6mm}
		& $d=2$, $q = 3\, (\mathrm{mod}\,4) $ & $\dfrac{q-1}{2}$ \\
		
		\phantom{$\dfrac{1^1}{1_{1_1}}$}\hspace{-6mm}
		& $d\geq 3$, $(d,q) \ne (3,2),(4,2),(4,3)$ & $\dfrac{q^d -q}{q-1}$ \\
		
		\phantom{$\dfrac{1^1}{1_{1_1}}$}\hspace{-6mm}
		& $d = 3$, $q=2$ & $3$ \\
		
		\phantom{$\dfrac{1^1}{1_{1_1}}$}\hspace{-6mm}
		& $d = 4$, $q=2$ & $7$ \\
		
		\phantom{$\dfrac{1^1}{1_{1_1}}$}\hspace{-6mm}
		& $d = 4$, $q=3$ & $26$ \\
		
		\phantom{$\dfrac{1^1}{1_{1_1}}$}\hspace{-6mm}
		$\mathrm{PSU}_d(q)$ & $d\geq 3$ odd & $\dfrac{q^d-q}{q+1}$ \\
		
		\phantom{$\dfrac{1^1}{1_{1_1}}$}\hspace{-6mm}
		& $d\geq 4$ even, $(d,q) \ne (4,2)$ & $\dfrac{q^d-1}{q+1} \leq \D_0(L) \leq \dfrac{q^d+q}{q+1}$ \\
		
		\phantom{$\dfrac{1^1}{1_{1_1}}$}\hspace{-6mm}
		& $d=4$, $q=2$ & $5$ \\
		
		\phantom{$\dfrac{1^1}{1_{1_1}}$}\hspace{-6mm}
		$\mathrm{PSp}_{2d}(q)$ & $d\geq 2$, $q$ even & $\dfrac{(q^d-1)(q^d-q)}{2(q+1)}$ \\
		
		\phantom{$\dfrac{1^1}{1_{1_1}}$}\hspace{-6mm}
		& $d\geq 2$, $q^d= 1\, (\mathrm{mod}\,4) $ & $\dfrac{q^d+1}{2}$ \\
		
		\phantom{$\dfrac{1^1}{1_{1_1}}$}\hspace{-6mm}
		& $d\geq 2$, $q^d= 3\, (\mathrm{mod}\,4) $ & $\dfrac{q^d-1}{2}$ \\
		
		\phantom{$\dfrac{1^1}{1_{1_1}}$}\hspace{-6mm}
		$\Omega_{2d+1}(q)$& $d\geq 3$, $q\ne 3$ and $2(q^2-1) \nmid q^{2d}-1$ & $\dfrac{q^{2d}-1}{q^2-1}$ \\
		
		\phantom{$\dfrac{1^1}{1_{1_1}}$}\hspace{-6mm}
		& $d\geq 3$, $q= 3$ or $2(q^2-1) \mid q^{2d}-1$ & $\dfrac{(q^{d}-1)(q^d-q)}{2(q+1)}$ \\
		
		\phantom{$\dfrac{1^1}{1_{1_1}}$}\hspace{-6mm}
		$\mathrm{P}\Omega^{-}_{2d}(q)$& $d\geq 4$ & $\dfrac{(q^d+1)(q^{d-1}-q)}{q^2-1}$ \\
		
		\phantom{$\dfrac{1^1}{1_{1_1}}$}\hspace{-6mm}
		$\mathrm{P}\Omega^{+}_{2d}(q)$& $d\geq 4$, $q\ne 2,3$ & $\dfrac{(q^d-1)(q^{d-1}+q)}{q^2-1}$ \\
		
		\phantom{$\dfrac{1^1}{1_{1_1}}$}\hspace{-6mm}
		& $d = 4$, $q= 2$ & $28$ \\
		
		\phantom{$\dfrac{1^1}{1_{1_1}}$}\hspace{-6mm}
		& $d \geq 5$, $q= 2$ or $d \geq 4$, $q= 3$  & $\dfrac{(q^d-1)(q^{d-1}-1)}{q^2-1}$  \\
		
		\hline
	\end{tabular}
	\medskip
	{\caption{$\D_0(L)$ for simple classical groups.}
		\label{table:ClassicOrdinary}}
\end{table}
}

{\small 

\begin{table}[ht]
	\centering
	\begin{tabular}{| c c | c c c |}
		\hline
		
		\phantom{$\dfrac{1^1}{1_{1_1}}$}\hspace{-6mm}
		$L$ & Comments & $\sim \log_q|L|$ & $\sim \log_q \D_0(L)$ & $\lim \limits_{q \rightarrow+\infty} \a(L)$ \\
		
		\hline \hline
		
		\phantom{$\dfrac{1^1}{1_{1_1}}$}\hspace{-6mm}
		$\PSL_d(q)$ & $d\geq 2$ & $d^2-1$ & $d-1$ & $\dfrac{1}{d+1}$ \\
		
		\phantom{$\dfrac{1^1}{1_{1_1}}$}\hspace{-6mm}
		$\mathrm{PSU}_d(q)$ & $d\geq 3$ & $d^2-1$ & $d-1$ & $\dfrac{1}{d+1}$ \\
		
		\phantom{$\dfrac{1^1}{1_{1_1}}$}\hspace{-6mm}
		$\mathrm{PSp}_{2d}(q)$ & $d\geq 2$, $q$ even & $2d^2+d$ & $2d-1$ & $\dfrac{2d-1}{2d^2+d}$ \\
		
		\phantom{$\dfrac{1^1}{1_{1_1}}$}\hspace{-6mm}
		& $d\geq 2$, $q$ odd & $2d^2+d$ & $d$ & $\dfrac{1}{2d+1}$ \\
		
		\phantom{$\dfrac{1^1}{1_{1_1}}$}\hspace{-6mm}
		$\Omega_{2d+1}(q)$& $d\geq 3$, $q$ odd & $2d^2+d$ & $2d-2$ & $\dfrac{2d-2}{2d^2+d}$ \\
		
		\phantom{$\dfrac{1^1}{1_{1_1}}$}\hspace{-6mm}
		$\mathrm{P}\Omega^{\pm}_{2d}(q)$& $d\geq 4$ & $2d^2-d$ & $2d-3$ & $\dfrac{2d-3}{2d^2-d}$ \\
		
		\hline
	\end{tabular}
	\medskip{\caption{Asymptotic quasirandom degree of simple classical groups.}
		\label{tableLieCl}}
\end{table}

\begin{table}[ht]
	\centering
	\begin{tabular}{| c c | c c c |}
		\hline
		
		\phantom{$\dfrac{1^1}{1_{1_1}}$}\hspace{-6mm}
		$L$ & Comments  & $\sim \log_q|L|$ & $\sim \log_q \D_0(L)$ & $\lim \limits_{q \rightarrow+\infty} \a(L)$ \\
		
		\hline \hline
		
		\phantom{$\dfrac{1^1}{1_{1_1}}$}\hspace{-6mm}
		$E_6(q)$ & & $78$ & $11$ & $\dfrac{11}{78}$ \\
		
		\phantom{$\dfrac{1^1}{1_{1_1}}$}\hspace{-6mm}
		$E_7(q)$ & & $133$ & $17$ & $\dfrac{17}{133}$  \\
		
		\phantom{$\dfrac{c1^1}{1_{1_1}}$}\hspace{-6mm}
		$E_8(q)$ & & $248$ & $29$ & $\dfrac{29}{248}$  \\
		
		\phantom{$\dfrac{1^1}{1_{1_1}}$}\hspace{-6mm}
		$F_4(q)$ & $q$ even & $52$ & $11$ & $\dfrac{11}{52}$ \\
		
		\phantom{$\dfrac{1^1}{1_{1_1}}$}\hspace{-6mm}
		$F_4(q)$ & $q$ odd & $52$ & $8$ & $\dfrac{2}{13}$ \\
		
		\phantom{$\dfrac{1^1}{1_{1_1}}$}\hspace{-6mm}
		$G_2(q)$ & $q \ne 3\, (\mathrm{mod}\,6) $ & $14$ & $3$ & $\dfrac{3}{14}$ \\
		
		\phantom{$\dfrac{1^1}{1_{1_1}}$}\hspace{-6mm}
		$G_2(q)$ & $q = 3\, (\mathrm{mod}\,6) $ & $14$ & $4$ & $\dfrac{2}{7}$ \\
		
		\phantom{$\dfrac{1^1}{1_{1_1}}$}\hspace{-6mm}
		$^2E_6(q^2)$ & & $78$ & $11$ & $\dfrac{11}{78}$ \\
		
		\phantom{$\dfrac{1^1}{1_{1_1}}$}\hspace{-6mm}
		$^3D_4(q^3)$ & & $28$ & $5$ & $\dfrac{5}{28}$  \\
		
		\phantom{$\dfrac{1^1}{1_{1_1}}$}\hspace{-6mm}
		$^2F_4(q^2)$ & $q=\sqrt{2^{2m+1}}$ & $52$ & $11$ & $\dfrac{11}{52}$  \\
		
		\phantom{$\dfrac{1^1}{1_{1_1}}$}\hspace{-6mm}
		$^2B_2(q^2)$ & $q=\sqrt{2^{2m+1}}$ & $10$ & $3$ &  $\dfrac{3}{10}$ \\
		
		\phantom{$\dfrac{1^1}{1_{1_1}}$}\hspace{-6mm}
		$^2G_2(q^2)$ & $q=\sqrt{3^{2m+1}}$ & $14$ & $4$ & $\dfrac{2}{7}$ \\
		
		\hline
	\end{tabular}
	\medskip
	{\caption{Asymptotic quasirandom degree of simple exceptional groups.}
	\label{tableLieEx}}
\end{table}
}

{\small

\begin{table}
	\centering
	\begin{tabular}{| c | c c c || c | c c c |}
		\hline
		
		\phantom{$\dfrac{1^1}{1_{1_1}}$}\hspace{-6mm}
		$L$ & $|L|$ & $\D_0(L)$ & $\a(L)$ & $L$ & $|L|$ & $\D_0(L)$ & $\a(L)$ \\
		
		\hline \hline
		
		\phantom{$\dfrac{1^1}{1_{1_1}}$}\hspace{-6mm}
		$M_{11}$ & $7\,920$ & $10$ & $\approx 0.2565$ &
		$Fi_{24}$ & $125\,520\,570\cdot 10^{16}$ & $8\,671$ & $\approx 0.1635$ \\
		
		\phantom{$\dfrac{1^1}{1_{1_1}}$}\hspace{-6mm}
		$M_{12}$ & $95\,040$ & $11$ & $\approx 0.2093$ &
		$HS$ & $44\,352\,000$ & $22$ & $\approx 0.1756$ \\
		
		\phantom{$\dfrac{1^1}{1_{1_1}}$}\hspace{-6mm}
		$M_{22}$ & $443\,520$ & $21$ & $\approx 0.2342$ &
		$McL$ & $898\,128\,000$ & $22$ & $\approx 0.1500$ \\
		
		\phantom{$\dfrac{1^1}{1_{1_1}}$}\hspace{-6mm}
		$M_{23}$ & $10\,200\,960$ & $22$ & $\approx 0.1916$ &
		$He$ & $4\,030\,387\,200$ & $51$ & $\approx 0.1778$ \\
		
		\phantom{$\dfrac{1^1}{1_{1_1}}$}\hspace{-6mm}
		$M_{24}$ & $244\,823\,040$ & $23$ & $\approx 0.1624$ & 
		$Ru$ & $145\,926\,144\cdot 10^3$ & $378$ & $\approx 0.2309$ \\
		
		\phantom{$\dfrac{1^1}{1_{1_1}}$}\hspace{-6mm}
		$J_1$ & $175\,560$ & $56$ & $\approx 0.3334$ &
		$Suz$ & $448\,345\,497 \cdot 10^3$ & $143$ & $\approx 0.1850$ \\
		
		\phantom{$\dfrac{1^1}{1_{1_1}}$}\hspace{-6mm}
		$J_2$ & $604\,800$ & $14$ & $\approx 0.1924$ &
		$O'N$ & $460\,815\,505\cdot 10^3$ & $10\,944$ & $\approx 0.3464$ \\
		
		\phantom{$\dfrac{1^1}{1_{1_1}}$}\hspace{-6mm}
		$J_3$ & $50\,232\,960$ & $85$ & $\approx 0.2506$ &
		$HN$ & $273\,030\,912\cdot 10^6$ & $133$ & $\approx 0.1472$ \\
		
		\phantom{$\dfrac{1^1}{1_{1_1}}$}\hspace{-6mm}
		$J_4$ & $867\,755\,710\cdot 10^{11}$ & $1\,333$ & $\approx 0.1568$ &
		$Ly$ & $517\,651\,790\cdot 10^8$ & $2\,480$ & $\approx 0.2031$ \\
		
		\phantom{$\dfrac{1^1}{1_{1_1}}$}\hspace{-6mm}
		$Co_3$ & $495\,766\,656\cdot 10^3$ & $23$ & $\approx 0.1165$ &
		$Th$ & $907\,459\,438\cdot 10^8$ & $248$ & $\approx 0.1413$ \\
		
		\phantom{$\dfrac{1^1}{1_{1_1}}$}\hspace{-6mm}
		$Co_2$ & $423\,054\,213\cdot 10^5$ & $23$ & $\approx 0.1000$ &
		$B$ & $415\,478\,148\cdot 10^{24}$ & $4\,371$& $\approx 0.1083$ \\
		
		\phantom{$\dfrac{1^1}{1_{1_1}}$}\hspace{-6mm}
		$Co_1$ & $415\,777\,680\cdot 10^{10}$ & $276$ & $\approx 0.1311$ &
		$M$ & $808\,017\,424\cdot 10^{44}$ & $196\,883$ & $\approx 0.0983$ \\
		
		\phantom{$\dfrac{1^1}{1_{1_1}}$}\hspace{-6mm}
		$Fi_{22}$ & $645\,617\,516\cdot 10^5$ & $78$ & $\approx 0.1371$ &
		$^2F_4(2)'$ & $17\,971\,200$ & $104$ & $\approx 0.2781$ \\
		
		\phantom{$\dfrac{1^1}{1_{1_1}}$}\hspace{-6mm}
		$Fi_{23}$ & $408\,947\,047\cdot 10^{10}$ & $782$ & $\approx 0.1555$ &
		$ $ & $ $ & $ $ & $ $ \\
		\hline
		
	\end{tabular}
	\medskip{\caption{Quasirandom degree of the sporadic simple groups.}
		\label{table:Spor}}
\end{table}
}

\section*{Acknowledgments}
We thank an anonymous referee who found a mistake in Lemma \ref{lemDP}.
Moreover, we thank S. Eberhard and P.H. Tiep for useful conversations.

\bibliographystyle{amsplain}



\begin{dajauthors}
\begin{authorinfo}[mb]
  Marco Barbieri\\
  Faculty of Mathematics and Physics, University of Ljubljana\\
  Jadranska 21, 1000 Ljubljana, Slovenia\\
  marco.barbieri\imageat{}fmf.uni-lj\imagedot{}si
\end{authorinfo}

\begin{authorinfo}[ls]
  Luca Sabatini\\
  Mathematics Institute, Zeeman Building, University of Warwick\\
  Coventry CV4 7AL, United Kingdom\\
  luca.sabatini\imageat{}warwick\imagedot{}ac\imagedot{}uk, sabatini.math\imageat{}gmail\imagedot{}com
\end{authorinfo}
\end{dajauthors}

\end{document}